\documentclass[a4paper,11pt,leqno]{article}

\usepackage[francais]{babel}
\usepackage[utf8]{inputenc}
\usepackage[T1]{fontenc}
\usepackage{amsthm}
\usepackage{amsmath}
\usepackage{amssymb}
\usepackage{mathrsfs}
\usepackage[all]{xy}
\usepackage{fmtcount}
\usepackage{titlesec}
\usepackage{tocloft}
\usepackage{titletoc}
\usepackage{tikz}
\usepackage{pdfpages}
\usepackage{mathtools}
\usepackage{appendix}

\usepackage{geometry}
\usepackage{color}
\usepackage{textcomp}
\usepackage{graphicx}
\usepackage{esint}
\usepackage{wrapfig}
\usepackage{float}
\usepackage{enumitem}

\usepackage{setspace}
\newtheorem{thm}{Th\'eor\`eme}[section]

\newtheorem{prop}[thm]{Proposition}

\newtheorem{lemma}[thm]{Lemme}

\newtheorem{corol}[thm]{Corollaire}
\newtheorem{csq}[thm]{Cons\'equence}

\theoremstyle{definition}
\newtheorem{defi}[thm]{D\'efinition}

\newtheorem{nota}[thm]{Notation}
\newtheorem{rem}[thm]{Remarque}

\theoremstyle{remark}

\titleformat{\section}
                [hang]
                {\center\normalfont\Large\bfseries}
                {\thesection .}
                {0.7em}
                {\bsc}
\titleformat{\subsection}
                [hang]
                {\normalfont\large\bfseries}
                {\thesection.\arabic{subsection}.}
                {0.7em}
                {}

\newcommand{\pic}{\mathrm{Pic}}

\newcommand{\br}{\mathrm{Br}}
\newcommand{\h}{\mathrm{H}}

\newcommand{\spec}{\mathrm{Spec}}

\usepackage{titlepic}

\usepackage{hyperref}
\hypersetup{
    colorlinks=true,
    linkcolor=blue,
    citecolor=brown
}

\begin{document}

\begin{center}
{\huge \bfseries Faisceaux $\mathbb{Q}/\mathbb{Z}(j)$ et \\ conjecture de Gersten \\ sur un corps imparfait}
\vspace{\baselineskip}

Alexandre \bsc{Lourdeaux}
\end{center}

\vspace{4cm}

On revoit explicitement la construction ainsi que certaines propri\'et\'es des complexes de faisceaux $\mathbb{Q}/\mathbb{Z}(j)$ sur certains sch\'emas. Plus pr\'ecis\'ement on consid\`ere les $k$-vari\'et\'es lisses et leurs anneaux locaux. Le but est d'avoir une r\'ef\'erence précise pour la conjecture de Gersten pour les faisceaux $\mathbb{Q}/\mathbb{Z}(j)$ sur un corps imparfait ainsi que pour la comparaison entre les groupes de cohomologie $\h^2(\ast,\mathbb{Q}/\mathbb{Z}(1))$ et $\h^2(\ast,\mathbb{G}_m)$.

On ne fait que pr\'esenter en d\'etail des choses connues des sp\'ecialistes.

\vspace{\baselineskip}

On note $p$ la caract\'eristique du corps de base $k$ et on la suppose non nulle. Pour un corps $K$, une \emph{$K$-vari\'et\'e} est un $K$-sch\'ema s\'epar\'e de type fini sur $K$.

	\section{Les "faisceaux" \'etales $\mathbb{Q}/\mathbb{Z}(j)$}
	\label{section faisceaux Q/Z}

	\subsection{Complexes de de\thinspace Rham-Witt logarithmiques}

\paragraph{Pro-complexe de de\thinspace Rham-Witt.}
Dans \cite{Illusie}, Illusie construit le \emph{pro-complexe de de\thinspace Rham-Witt} pour tout topos annel\'e en $\mathbb{F}_p$-alg\`ebres. Les topos qui nous int\'eressent seront les petits et grands sites \'etales de $k$-sch\'emas $X$ munis des faisceaux $O_X=\mathbb{G}_{a,X}$. On note $X_{\text{\'Et}}$ le grand site \'etale du sch\'ema $X$ et $X_{\text{\'et}}$ son petit site \'etale.

Pour synth\'etiser, un \emph{V-pro-complexe de de\thinspace Rham-Witt} $M_\bullet^\bullet$ sur un site $C$ correspond \`a la donn\'ee \begin{itemize}
\item d'une famille d'alg\`ebres diff\'erentielles gradu\'ees (strictement anticommutatives) $M^\bullet_n$ sur $C$ pour $n \in \mathbb{Z}$, avec \begin{itemize}
\item $M_n^\bullet=0$ pour tout $n \leqslant 0$,
\item $M_1^0$ une $\mathbb{F}_p$-alg\`ebre,
\item et $\forall n \geqslant 1, \,  M^0_n=\mathrm{W}_n(M^0_1)$ (l'anneau des $n$-vecteurs de Witt de $M_1^0$);
\end{itemize}
\item de morphismes d'alg\`ebres diff\'erentielles gradu\'ees $\mathrm{R}=\mathrm{R}_n^\bullet : M^\bullet_{n+1} \to M^\bullet_n$;
\item d'applications additives $\mathrm{V}=\mathrm{V}_{n}^i : M^i_n \to M^i_{n+1}$;
\end{itemize} ces donn\'ees v\'erifiant certains axiomes (voir \cite[Déf. I.1.1 p. 543]{Illusie}).

Il est possible d'\'etendre naturellement toute $\mathbb{F}_p$-alg\`ebre $A$ sur un site $C$ en un V-prop-complexe de De\thinspace Rham-Witt :
\begin{thm}[{\cite[Th. I.1.3, page 544]{Illusie}}]
\label{thm de rham witt}
Soit $C$ un site. Le foncteur de la cat\'egorie des V-pro-complexes de de\thinspace Rham-Witt sur $C$ dans la cat\'egorie des $\mathbb{F}_p$-alg\`ebres sur $C$, qui \`a $M_\bullet^\bullet$ associe $M^0_1$, admet un adjoint \`a gauche $A \mapsto \mathrm{W}_\bullet \Omega_{(C,A)}^\bullet$.

De plus, pour toute $\mathbb{F}_p$-alg\`ebre $A$ sur $C$, on a un isomorphisme canonique \[ \Omega_A^\bullet \overset{\sim}{\to}\mathrm{W}_1 \Omega_{(C,A)}^\bullet , \] l'alg\`ebre des diff\'erentielles de $A$ relativement au faisceau constant $\mathbb{F}_p$.
\end{thm}

%Si l'on veut sp\'ecifier le topos $X$, on notera $\mathrm{W}_\bullet \Omega_{(X,A)}^\bullet$ pour $\mathrm{W}_\bullet \Omega_A^\bullet$.

\begin{defi} \begin{itemize}
\item Pour $A$ une $\mathbb{F}_p$-alg\`ebre sur le site $C$, on appelle \emph{pro-complexe de De\thinspace Rham-Witt de $A$ (sur $C$)} le V-pro-complexe de De\thinspace Rham Witt $\mathrm{W}_\bullet \Omega_{(C,A)}^\bullet$.
%\item Si $(X,O_X)$ est un topos annel\'e en $\mathbb{F}_p$-alg\`ebre, on d\'esigne par $\mathrm{W}_\bullet \Omega_X^\bullet$ le pro-complexe de De \thinspace Rham-Witt de $O_X$ sur $X$. C'est le \emph{pro-complexe de De \thinspace Rham-Witt de $X$ (sous-entendu \og de $(X,O_X)$ \fg{})}.
%\item Si $X$ est un $\mathbb{F}_p$-sch\'ema, on appelle \emph{(grand) pro-complexe de De \thinspace Rham-Witt de $X$}, et noter $\mathrm{W}_\bullet \Omega_X^\bullet$, le pro-complexe de De \thinspace Rham-Witt du topos annel\'e $(X_{\text{\'Et}},O_X)$, avec $X_{\text{\'Et}}$ le grand site \'etale de $X$.
%; et on va appel\'e \emph{petit pro-complexe de De \thinspace Rham-Witt de $X$} et noter $\mathrm{w}_\bullet \Omega_X^\bullet$, le pro-complexe de De \thinspace Rham-Witt du topos annel\'e $(X_{\text{\'et}},O_X)$.
\item Soit $X$ un $\mathbb{F}_p$-sch\'ema. Le \emph{pro-complexe de De\thinspace Rham-Witt de $\mathcal{F}$ (sur $X$)} d'un faisceau de $\mathbb{F}_p$-alg\`ebres $\mathcal{F}$ sur $X_{\text{\'Et}}$ (resp. sur $X_{\text{\'et}}$) est le pro-complexe de De\thinspace Rham-Witt $\mathrm{W}_\bullet \Omega_{(X_{\text{\'Et}},\mathcal{F})}^\bullet$ (resp. $\mathrm{W}_\bullet \Omega_{(X_{\text{\'et}},\mathcal{F})}^\bullet$).

Le \emph{grand pro-complexe de De\thinspace Rham-Witt de $X$} est le pro-complexe de De\thinspace Rham-Witt $\mathrm{W}_\bullet \Omega_{(X_{\text{\'Et}},O_X)}^\bullet$, not\'e $\mathrm{W}_\bullet \Omega_{X_{\text{\'Et}}}^\bullet$. Le \emph{(petit) pro-complexe de De\thinspace Rham-Witt de $X$} est le pro-complexe de De\thinspace Rham-Witt $\mathrm{W}_\bullet \Omega_{(X_{\text{\'et}},O_X)}^\bullet$, not\'e $\mathrm{W}_\bullet \Omega_{X_{\text{\'et}}}^\bullet$.
\end{itemize}
\end{defi}

Sur un site $C$, un morphisme de $\mathbb{F}_p$-alg\`ebres $u : A \to B$ induit un morphisme de pro-complexes de De\thinspace Rham-Witt \[ \mathrm{W}_\bullet \Omega_u^\bullet : \mathrm{W}_\bullet \Omega_{(C,A)}^\bullet \to \mathrm{W}_\bullet \Omega_{(C,B)}^\bullet .\] Et si $f : D \to C$ est un morphisme de sites\footnote{pour rappel, $f$ est donn\'e par un foncteur $C \to D$} et si $A$ est un faisceau sur $C$ et $B$ un faisceau sur $D$, alors $f_\ast \mathrm{W}_\bullet \Omega_{(D,B)}^\bullet$ et $f^{-1} \mathrm{W}_\bullet \Omega_{(C,A)}^\bullet$ sont des V-pro-complexes de De\thinspace Rham-Witt (sur $C$ et $D$ respectivement) et on a des morphismes canoniques \begin{gather*}
\mathrm{W}_\bullet \Omega_{(C,f_\ast B)}^\bullet  \to f_\ast \mathrm{W}_\bullet \Omega_{(D,B)}^\bullet , \\
f^{-1} \mathrm{W}_\bullet \Omega_{(C,A)}^\bullet \to \mathrm{W}_\bullet \Omega_{(D,f^{-1} A)}^\bullet .
\end{gather*}

\begin{rem}
Soit $C$ un site et soit $c_0$ un objet de $C$. On consid\`ere le site $C^\prime$ de cat\'egorie sous-jacente $C/c_0$ des objets au-dessus de $c_0$ et dont la topologie est issue de $C$. On a un foncteur continu \[ \left\lbrace \begin{array}{ccc} 
C^\prime & \to & C \\
(c \to c_0) & \mapsto & c
\end{array} \right.\] Alors pour toute $\mathbb{F}_p$-alg\`ebre $A$ sur $C$, en appelant $A^\prime$ sa restriction $A \circ u$, on a une identification $\left( \mathrm{W}_\bullet \Omega_{(C,A)}^\bullet \right)_{|C^\prime} = \mathrm{W}_\bullet \Omega_{(C^\prime,A^\prime)}^\bullet$.

En particulier, si $X $ est un $\mathbb{F}_p$-sch\'emas, la restriction de $\mathrm{W}_\bullet \Omega_{(X_{\text{\'Et}},O_X)}^\bullet$ au petit site \'etale $X_{\text{\'et}}$ de $X$ est $\mathrm{W}_\bullet \Omega_{(X_{\text{\'et}},O_X)}^\bullet$. Et si $X \to Y$ est un morphisme de $\mathbb{F}_p$-sch\'emas, on a $\left( \mathrm{W}_\bullet \Omega_{(Y_{\text{\'Et}},O_Y)}^\bullet \right)_{|X_{\text{\'Et}}} = \mathrm{W}_\bullet \Omega_{(X_{\text{\'Et}},O_X)}^\bullet$.
\end{rem}

\begin{prop}
\label{propriétés RW}
\begin{enumerate}
\item (\cite[\S I.1.10]{Illusie}) \label{limite inductive RW} Soit $C$ un site et $A$ une $\mathbb{F}_p$-alg\`ebre sur $C$ d\'efinie comme limite \emph{inductive filtrante} de $\mathbb{F}_p$-alg\`ebres $A_i$. Alors $\underset{\longrightarrow}{\mathrm{lim}} \, \mathrm{W}_\bullet \Omega_{A_i}^\bullet$ existe dans la cat\'egorie des V-pro-complexes de De\thinspace Rham-Witt et le morphisme canonique \[ \underset{\longrightarrow}{\mathrm{lim}} \, \mathrm{W}_\bullet \Omega_{A_i}^\bullet \to \mathrm{W}_\bullet \Omega_{A}^\bullet \] (induit par les $A_i \to \underset{\longrightarrow}{\mathrm{lim}} A_i=A$) est un isomorphisme.
\item (\cite[\S I.1.12]{Illusie}) \label{image inverse RW} Si $f : D \to C$ est un morphisme de sites, et si $A$ est une $\mathbb{F}_p$-alg\`ebre sur $C$, alors le morphisme canonique \[ f^{-1} \mathrm{W}_\bullet \Omega_{(C,A)}^\bullet \to \mathrm{W}_\bullet \Omega_{(D,f^{-1} A)}^\bullet \] est un isomorphisme.
%\item  (\cite[Prop. I.1.14]{Illusie}) Si $f : X \to Y$ est un morphisme \'etale entre $\mathbb{F}_p$-sch\'emas, alors \[ f^\ast \mathrm{W}_\bullet \Omega_Y^\bullet \to \mathrm{W}_\bullet \Omega_X^\bullet \] est un isomorphisme.
\end{enumerate}
\end{prop}

\begin{rem}
\label{lemme limite RW}
Soit $X$ un $\mathbb{F}_p$-sch\'ema qui est une limite projective de $\mathbb{F}_p$-sch\'ema $X_\lambda$ pour $\lambda$ d\'ecrivant un ensemble pr\'eordonn\'e filtrant $\Lambda$. Appelons $f_\lambda : X \to X_\lambda$ les morphismes de $\mathbb{F}_p$-sch\'emas correspondant \`a la limite projective. D\'esignons par $\tau$ le grand site \'etale $\text{\'Et}$ ou le petit site \'etale $\text{\'et}$, et soit $(\mathcal{F}_\lambda)_{\lambda \in \Lambda}$ un syst\`eme inductif de faisceaux en $\mathbb{F}_p$-alg\`ebres avec $\mathcal{F}_\lambda$ d\'efini sur $X_{\lambda,\tau}$ pour chaque indice $\lambda$. On note $\mathcal{F}$ la limite des $f_\lambda^{-1} \mathcal{F}_\lambda$; c'est une $\mathbb{F}_p$-alg\`ebre sur $X_{\tau}$.

Le point \ref{image inverse RW} de la proposition implique qu'on a des isomorphismes \[ f_\lambda^{-1} \mathrm{W}_\bullet \Omega^\bullet_{(X_{\lambda,\tau}, \mathcal{F}_\lambda)} \overset{\sim}{\to} \mathrm{W}_\bullet \Omega^\bullet_{(X_\tau,f_\lambda^{-1} \mathcal{F}_\lambda)} \] de V-pro-complexes sur $X_{\tau}$. On en d\'eduit un isomorphisme \[\underset{\longrightarrow}{\mathrm{lim}} \, f_\lambda^{-1} \mathrm{W}_\bullet \Omega^\bullet_{(X_{\lambda,\tau},\mathcal{F}_\lambda)} \overset{\sim}{\to} \underset{\longrightarrow}{\mathrm{lim}} \, \mathrm{W}_\bullet \Omega^\bullet_{(X_\tau,f_\lambda^{-1} \mathcal{F}_\lambda)} . \] Le point \ref{limite inductive RW} assure alors qu'on a un isomorphisme  \[\underset{\longrightarrow}{\mathrm{lim}} \, f_\lambda^{-1} \mathrm{W}_\bullet \Omega^\bullet_{(X_{\lambda,\tau},\mathcal{F}_\lambda)} \overset{\sim}{\to}  \mathrm{W}_\bullet \Omega^\bullet_{(X_\tau,\mathcal{F})} \] car $\underset{\longrightarrow}{\mathrm{lim}} \, f_\lambda^{-1} \mathcal{F}_\lambda \overset{\sim}{\to} \mathcal{F}$ comme faisceaux sur $X_{\tau}$.

Cela s'applique au cas particulier o\`u $(A_\lambda)_\lambda$ une famille directe filtrante de $\mathbb{F}_p$-alg\`ebres avec \begin{gather*}
X_\lambda=\spec(A_\lambda) , \\
X= \underset{\longleftarrow}{\mathrm{lim}} \, \spec(A_\lambda) = \spec(\underset{\longrightarrow}{\mathrm{lim}} \, A_\lambda) , \\
\mathcal{F}_\lambda=O_{\spec(A_\lambda)} \text{ et } \mathcal{F}=O_{X} . 
\end{gather*}
\end{rem}

\paragraph{$\triangleright$ \bsc{Pro-complexe de De\thinspace Rham-Witt logarithmique.}}

	\subparagraph{D\'efinition locale.}
Soit $X$ un sch\'ema sur $\mathbb{F}_p$. Par d\'efinition des V-pro-complexes de De\thinspace Rham-Witt, $\mathrm{W}_n \Omega_{(X_{\text{\'et}},O_X)}^\bullet$ est une alg\`ebre diff\'erentielle gradu\'ee sur $X_{\text{\'et}}$ pour tout entier $n$, et $\mathrm{W}_n \Omega_{(X_{\text{\'et}},O_X)}^0$ est le faisceau des $n$-vecteurs de Witt de $O_X$ pour tout $n \geqslant 1$. Ainsi, \begin{itemize}
\item tout $x \in O_X$ a son \emph{repr\'esentant multiplicatif} $\underline{x} \in W_n(O_X)$,
\item tout $y \in W_n(O_X)$ a une image $\mathrm{d} y \in \mathrm{W}_n \Omega_{(X_{\text{\'et}},O_X)}^1$ par la diff\'erentielle $\mathrm{d}$,
\item et pour toute famille $\alpha_1,\cdots,\alpha_r$ avec $\alpha_i \in\mathrm{W}_n \Omega_{(X_{\text{\'et}},O_X)}^{j_i}$, on peut regarder le produit $\alpha_1 \cdots \alpha_r \in \mathrm{W}_n \Omega_{(X_{\text{\'et}},O_X)}^{j_1+\cdots+j_r}$.
\end{itemize} On consid\`ere alors les morphismes de faisceaux en groupes ab\'eliens sur $X_{\text{\'et}}$, pour $j \geqslant 1$ : \[  \left\lbrace \begin{array}{ccc} (O_X^\times)^{\otimes j} & \to & \mathrm{W}_n \Omega_{(X_{\text{\'et}},O_X)}^j \\
x_1 \otimes \cdots \otimes x_j & \mapsto & \mathrm{d}\underline{x_1} / \underline{x_1} \cdots \mathrm{d}\underline{x_j} / \underline{x_j}
\end{array} \right.,\] et on note $\nu_n^j(X)$ le faisceau image. En particulier on dispose du morphisme \[ \mathrm{d} \log : O_X^\times \to \mathrm{W}_n \Omega_{X_{\text{\'et}}}^1 \]

En voyant $\nu_n^j(X)$ comme le complexe $\cdots \to 0 \to \nu_n^j(X) \to 0 \to \cdots$ concentr\'e en degr\'e $0$, on d\'efinit les complexes \[ \mathbb{Z}/{p^n\mathbb{Z}}(j)_X:= \nu_n^j(X)[-j] \text{ et } \mathbb{Q}_p/{\mathbb{Z}_p}(j)_X:= \underset{\underset{n}{\longrightarrow}}{\mathrm{lim}} \; \nu_n^j(X) [-j]  \] la limite \'etant d\'efinie via les restrictions des morphismes $V_n^j : \mathrm{W}_n \Omega_{(X_{\text{\'et}},O_X)}^j \to \mathrm{W}_{n+1} \Omega_{(X_{\text{\'et}},O_X)}^j$ qui font partie de la d\'efinition de V-pro-complexe de $\mathrm{W}_\bullet \Omega_{(X_{\text{\'et}},O_X)}^\bullet$.

\begin{prop}
\label{suite dlog}
Si $X$ est un sch\'ema r\'egulier sur $\mathbb{F}_p$ avec $p$ premier, alors pour tout $n \geqslant 1$ on a une suite exacte \begin{equation} 
\label{suite exacte logarithmique} 0 \to O_X^\times \overset{(\cdot)^{p^n}}{\to} O_X^\times \overset{\mathrm{d} \log}{\to} \nu_n^1(X) \to 0 . \end{equation} de faisceaux sur $X_{\text{\'et}}$, le \emph{petit} site \'etale de $X$.
\end{prop}

\begin{proof}
$\diamond$ La surjectivit\'e de $O_X^\times \overset{\mathrm{d} \log}{\to} \nu_n^1(X)$ vient de la d\'efinition de $\nu_n^1$.

$\diamond$ L'injectivit\'e de $O_X^\times \overset{(\cdot)^{p^n}}{\to} O_X^\times$ provient du fait que $X$ est r\'eduit. En effet, pour tout morphisme \'etale $Y \to Z$, $Y$ est r\'eduit si, et seulement si, $Z$ est r\'eduit (\cite[\S 2.3, Prop. 9]{BLR}). Ainsi, tout $X$-sch\'ema \'etale $Y$ est r\'eduit, donc $O_Y(Y)^\times \overset{(\cdot)^p}{\to} O_Y(Y)^\times$ est un morphisme de groupes injectif ($O_Y(Y)$ \'etant une $\mathbb{F}_p$-alg\`ebre).

Si $X$ est \emph{lisse} sur $\mathbb{F}_p$, l'exactitude de \[ O_X^\times \overset{(\cdot)^{p^n}}{\to} O_X^\times \overset{\mathrm{d} \log}{\to} \mathrm{W}_n \Omega^1_{X_{\text{\'et}}} \] est exactement \cite[Prop I.3.23.2, p. 580]{Illusie}. Sinon, on se ram\`ene au cas lisse gr\^ace au th\'eor\`eme de Popescu \cite[Th. 2.5]{Popescu2} comme suit. D\'ej\`a, on peut supposer $X$ affine noeth\'erien \footnote{Par d\'efinition d'un $\mathbb{F}_p$-sch\'ema r\'egulier, les fibres du morphisme structural $X \to \spec(\mathbb{F}_p)$ sont localement noeth\'eriennes.}, $X=\spec (A)$ pour une $\mathbb{F}_p$-alg\`ebre r\'eguli\`ere $A$ qui est aussi un anneau noeth\'erien. Alors $\mathbb{F}_p \to A$ est un morphisme r\'egulier d'anneaux noeth\'eriens, donc d'apr\`es le th\'eor\`eme de Popescu, $A$ est une limite inductive filtrante de $\mathbb{F}_p$-alg\`ebres lisses de type fini $A_\lambda$. Selon l'exactitude dans le cas lisse sur $\mathbb{F}_p$, les suites de faisceaux \[ O_{\spec(A_\lambda)}^\times \overset{(\cdot)^{p^n}}{\to} O_{\spec(A_\lambda)}^\times \overset{\mathrm{d} \log}{\to} \mathrm{W}_n \Omega^1_{\spec(A_\lambda)_{\text{\'et}}} \] sont exactes. En notant $f_\lambda$ les morphismes $\spec(A) \to \spec(A_\lambda)$ tels que $\spec(A) \to \underset{\underset{\lambda}{\longleftarrow}}{\mathrm{lim}} \; \spec(A_\lambda)$ est un isomorphisme, il vient que les suites de faisceaux (sur $\spec(A)$) \[ f_\lambda^{-1} O_{\spec(A_\lambda)}^\times \overset{(\cdot)^{p^n}}{\to} f_\lambda^{-1} O_{\spec(A_\lambda)}^\times \overset{\mathrm{d} \log}{\to} f_\lambda^{-1} \mathrm{W}_n \Omega^1_{\spec(A_\lambda)_{\text{\'et}}} \] sont exactes. Or $\underset{\underset{\lambda}{\longrightarrow}}{\mathrm{lim}} \; f_\lambda^{-1} O_{\spec(A_\lambda)} \overset{\sim}{\to} O_{spec(A)}$, donc $\underset{\underset{\lambda}{\longrightarrow}}{\mathrm{lim}} \; f_\lambda^{-1} O_{\spec(A_\lambda)}^\ast \overset{\sim}{\to} O_{spec(A)}^\ast$ et d'apr\`es la remarque \ref{lemme limite RW}, on a $\underset{\underset{\lambda}{\longrightarrow}}{\mathrm{lim}} \; f_\lambda^{-1} \mathrm{W}_n \Omega^1_{\spec(A_\lambda)_{\text{\'et}}} = \mathrm{W}_n \Omega^1_{\spec(A)_{\text{\'et}}}$.
\end{proof}

	\subparagraph{D\'efinition globale.}
Pour $j \geqslant 1$, on peut \'egalement d\'efinir de la m\^eme fa\c{c}on le faisceau $\nu_n^j$ et les complexes $\mathbb{Z}/p\mathbb{Z}(j)$, $\mathbb{Q}_p/\mathbb{Z}_p(j)$ sur le grand site \'etale de $\spec(\mathbb{F}_p)$ \`a partir de $\mathrm{W}_\bullet \Omega^\bullet_{\spec(\mathbb{F}_p)_{\text{\'Et}}}$. On a alors, pour tout $\mathbb{F}_p$-sch\'ema $X$, \begin{gather*}
\left( \nu_n^j \right)_{|X_{\text{\'et}}} = \nu_n^j(X) , \\
\mathbb{Z}/p\mathbb{Z}(j)_{| X_{\text{\'et}}} = \mathbb{Z}/p\mathbb{Z}(j)_X , \\
\mathbb{Q}_p/\mathbb{Z}_p(j)_{|X_\text{\'et}} = \mathbb{Q}_p/\mathbb{Z}_p(j)_X
\end{gather*} les restrictions \'etant induites par rapport au foncteur \[ \left\lbrace \begin{array}{ccc}
X_{\text{\'et}} & \to & \spec(\mathbb{F}_p)_{\text{\'Et}} \\
(U \to X) & \mapsto & U
\end{array} \right. . \]

\paragraph{$\triangleright$ \bsc{Faisceaux \`a la Kato.}} On d\'efinit pour $j \geqslant 1$ les faisceaux sur le \emph{grand} site \'etale de $\spec(\mathbb{F}_p)$ : \[ \mathbb{Q}_q/{\mathbb{Z}_q}(j):= \underset{\underset{n}{\longrightarrow}}{\mathrm{lim}} \; \mu_{q^n}^{\otimes j} \text{ pour } q \text{ premier } \neq p \] o\`u $\mu_{q^n}$ est le groupe des racines $q^n$-i\`emes de l'unit\'e; et on d\'efinit le complexe de faisceaux \[ \mathbb{Q}/{\mathbb{Z}}(j):= \left( \bigoplus_{q \neq p \text{ premier}} \mathbb{Q}_q/{\mathbb{Z}_q}(j) \right) \oplus \mathbb{Q}_p/{\mathbb{Z}_p}(j) .  \] Pour le corps $\mathbb{Q}$ des nombres rationnels, on d\'efinit de m\^eme $\mathbb{Q}_q/\mathbb{Z}_q(j):=\underset{\underset{n}{\longrightarrow}}{\mathrm{lim}} \; \mu_{q^n}^{\otimes j}$ pour tout nombre premier $q$ et on pose \[ \mathbb{Q}/\mathbb{Z}(j):= \bigoplus_{q \text{ premier}} \mathbb{Q}_q/\mathbb{Z}_q(j) . \] On d\'efinit enfin $\mathbb{Q}/\mathbb{Z}(0)$ comme le faisceau constant $\mathbb{Q}/\mathbb{Z}$ sur le grand site \'etale de $\mathbb{F}_p$ ou de $\mathbb{Q}$.

Introduisons la notation utile suivante qu'on rencontrera par la suite. Si $d$ et $j$ sont des entiers positifs, pour tout $\mathbb{F}_p$-sch\'ema $X$ (ou tout sch\'ema $X$ sur un corps de caract\'eristique $p$) on note $\mathcal{H}^d_{(X)}(\mathbb{Q}/\mathbb{Z}(j))$ le faisceau sur le petit site de Zariski de $X$ associ\'e au pr\'efaisceau $U \mapsto \h^d(U,\mathbb{Q}/\mathbb{Z}(j))$.

\begin{lemma}
\label{lemme RW inverse}
Soit $X$ un $\mathbb{F}_p$-sch\'ema qui est une limite projective de $\mathbb{F}_p$-sch\'emas $X_\lambda$ pour $\lambda$ d\'ecrivant un ensemble pr\'eordonn\'e filtrant $\Lambda$. On suppose que les morphismes de transitions entre les $X_\lambda$ sont affines et on note $f_\lambda : X \to X_\lambda$ les morphismes canoniques. Alors pour tout $n \geqslant 0$ et $j \geqslant 1$, on a canoniquement des isomorphismes \[ \underset{\underset{\lambda}{\longrightarrow}}{\mathrm{lim}} \, f_\lambda^{-1} (\nu_n^j(X_\lambda)) \overset{\sim}{\to} \nu_n^j(X) \] et \[ \underset{\underset{\lambda}{\longrightarrow}}{\mathrm{lim}} \, f_\lambda^{-1} \mathbb{Q}/\mathbb{Z}(j)_{X_\lambda} \overset{\sim}{\to} \mathbb{Q}/\mathbb{Z}(j)_{X} . \]
\end{lemma}

\begin{proof}
L'isomorphisme du lemme pour $\nu_n^j$ vaut en raison de la remarque \ref{lemme limite RW}.

Pour $\mathbb{Q}/\mathbb{Z}(j)$ on distingue les parties de torsions $q$-primaires pour les diff\'erents nombres premiers $q$. Le cas de la torsion $p$-primaire se ram\`ene au cas de $\nu_n^j$. La $q$-torsion pour $q$ premier $\neq p$ r\'esulte du fait que pour tout entier $n \geqslant 0$, on a $ \underset{\longrightarrow}{\mathrm{lim}} \, f_\lambda^{-1} (\mu_{q^n}^{\otimes j})_{|(X_\lambda)_{\text{\'et}}} \overset{\sim}{\to}  (\mu_{q^n}^{\otimes j})_{|X_{\text{\'et}}}$.
\end{proof}

	\subsection{Conjecture de Gersten}

On fixe un entier positif $d$.

\begin{nota} \begin{enumerate}
\item Pour tout faisceau \'etale $\mathcal{F}$ sur un $k$-sch\'ema $U$ et tout ferm\'e $Z$ de $U$, $\h^d_Z(U,\mathcal{F})$ d\'esigne le groupe de cohomologie de $U$ \`a coefficients dans $\mathcal{F}$ et \`a support dans $Z$.
\item Pour $\mathcal{F}$ un faisceau \'etale sur un $k$-sch\'ema $X$, on note, pour tout $x \in X$ (point ensembliste), $\h^d_x (X,\mathcal{F})$ la limite inductive $\underset{\longrightarrow}{\mathrm{lim}} \; \h^d_{U \cap \overline{\{ x \} }} (U, \mathcal{F}_{\mid U})$ portant sur les ouverts $U$ de $X$ (pour la topologie de Zariski) contenant $x$.
\item Enfin, pour un point $x$ d'un sch\'ema $X$, on note $i_x$ le morphisme de sch\'emas $\spec(\kappa(x)) \to X$ et $(i_x)_\ast$ est le morphisme qui transporte les faisceaux sur $\spec(\kappa(x))$ vers $X$.
\end{enumerate}
\end{nota}

Comme expliqu\'e dans \cite[\S 1.1]{CT-H-K}, pour tout faisceau $\mathcal{F}$ sur le petit site \'etale d'un sch\'ema $X$ \'equidimensionnel \footnote{implicitement de dimension de Krull finie} et noeth\'erien, on a la suite spectrale convergente de coniveau \begin{equation}
\label{suite spectrale coniveau} E_1^{p,q} = \bigoplus_{x \in X^{(p)}} \h_x^{p+q}(X,\mathcal{F}) \Rightarrow \h^{p+q}(X,\mathcal{F}) ,
\end{equation} dont on tire le complexe de Cousin de faisceaux zariskiens sur $X$ : \begin{equation}
\label{complexe Cousin}
 0 \to \bigoplus_{x \in X^{(0)}} (i_x)_\ast \h_x^d(X,\mathcal{F}) \to \bigoplus_{x \in X^{(1)}} (i_x)_\ast \h_x^{d+1}(X,\mathcal{F}) \to \cdots  \to \bigoplus_{x \in X^{(d+r)}} (i_x)_\ast \h_x^{d+r}(X,\mathcal{F}) \to \cdots .
\end{equation}

On dit que l\emph{a conjecture de Gersten vaut (en degr\'e $d$) pour $X$ et $\mathcal{F}$} lorsque le complexe \eqref{complexe Cousin} est une r\'esolution (flasque) du faisceau $\mathcal{H}^d_{(X)}(\mathcal{F})$ sur $X$ pour la topologie de Zariski obtenu en faisceautisant le pr\'efaisceau $U \mapsto \h^d(U,\mathcal{F})$. Cela signifie que la suite exacte de faisceaux zariskiens suivante est exacte : \begin{multline}
\label{complexe gersten}
 0 \to \mathcal{H}^d_{(X)}(\mathcal{F}) \to \bigoplus_{x \in X^{(0)}} (i_x)_\ast \h_x^d(X,\mathcal{F}) \to \bigoplus_{x \in X^{(1)}} (i_x)_\ast\h_x^{d+1}(X,\mathcal{F}) \to \cdots \\
 \cdots \to \bigoplus_{x \in X^{(d+r)}} (i_x)_\ast \h_x^{d+r}(X,\mathcal{F}) \to \cdots .
\end{multline} On note ce complexe $G^d(X,\mathcal{F})$ qu'il soit exact ou non.

Remarquons que lorsque $X$ est de plus int\`egre, le faisceau $\bigoplus_{x \in X^{(0)}} (i_x)_\ast \h_x^d(X,\mathcal{F})$ est plus simplement le faisceau $(i_\eta)_\ast \h^d(K(X),\mathcal{F})$ o\`u $\eta$ d\'esigne le point g\'en\'erique de $X$ et $K(X)$ son corps de fonctions.

\begin{rem}
\label{rem fonctarialite gersten}
Au vu de la construction de \cite[\S 1.1]{CT-H-K}, la suite spectrale \eqref{suite spectrale coniveau} ci-dessus (et donc aussi les complexes \eqref{complexe Cousin} et \eqref{complexe gersten}) est fonctorielle contravariante pour les morphismes plats; tout morphisme plat $f : Y \to X$ induit un morphisme de complexes $ f^{-1} G(X,\mathcal{F}) \to G(Y,f^{-1} \mathcal{F})$ naturel en le faisceau $\mathcal{F}$ sur $X$. Dans la suite on s'int\'eresse aux immersions ouvertes $U \hookrightarrow X$ et aux anneaux locaux $\spec(O_{X,x_0}) \to X$ pour $x_0 \in X$.
\end{rem}

\begin{lemma}
\label{lemma fibres gersten}
Soit $X$ un sch\'ema \'equidimensionnel noeth\'erien et soit $x_0 \in X$ (on note $Y= \spec(O_{X,x_0})$ et $i$ le morphisme $\spec(O_{X,x_0}) \to X$). Soit $\mathcal{F}$ un faisceau \'etale sur $X$. On a un morphisme de complexes $\phi : G^d(X,\mathcal{F}) \to G^d(Y, i^{-1} \mathcal{F})$ (rem. \ref{rem fonctarialite gersten}). Alors $\phi$ induit un isomorphisme \[ G^d(X,\mathcal{F})_{x_0} \overset{\sim}{\to} G^d(Y, i^{-1} \mathcal{F})(Y) . \] C'est-\`a-dire que la fibre de $G(X,\mathcal{F})$ en $x_0$ s'identifie aux sections globales de $G(Y, i^{-1} \mathcal{F})$.
\end{lemma}

\begin{proof}
Notons tout d'abord que pour tout $x \in X$ et tout ouvert $U$ de $X$ contenant $x$, le morphisme canonique $\h^d_x(X,\mathcal{F}) \to \h^d_x(U,\mathcal{F}_{|U})$ est un isomorphisme, et que pour tout $x \in Y$, le morphisme canonique $\h^d_x(X,\mathcal{F}) \to \h^d_x(Y,i^{-1}\mathcal{F})$ est aussi un isomorphisme.

Pour tout ouvert $U$ de $X$ contenant $x_0$, on a un morphisme de complexes $G(U,\mathcal{F}) \to G(Y,i^{-1} \mathcal{F})$ dont on tire, en prenant les sections globales, le diagramme commutatif \[ \xymatrix{
0 \ar[r] & \mathcal{H}^d_{(U)}(\mathcal{F})(U) \ar[r] \ar[d] & \bigoplus_{x \in U^{(0)}} \h_x^d(U,\mathcal{F}) \ar[r] \ar[d] & \bigoplus_{x \in U^{(1)}} \h_x^{d+1}(U,\mathcal{F}) \ar[r] \ar[d] & \cdots \\
0 \ar[r] & \mathcal{H}^d_{(Y)}(i^{-1} \mathcal{F})(Y) \ar[r] & \bigoplus_{x \in Y^{(0)}}\h_x^d(Y, i^{-1}\mathcal{F}) \ar[r] & \bigoplus_{x \in Y^{(1)}} \h_x^{d+1}(Y,i^{-1} \mathcal{F}) \ar[r] & \cdots
} \] qui se r\'ecrit, compte tenu de ce qui a \'et\'e dit ci-dessus, et en constatant que $\mathcal{H}^d_{(Y)}(i^{-1} \mathcal{F})(Y)=\h^d(Y,i^{-1} \mathcal{F})$ (car $Y$ est le spectre d'un anneau local) : \[ \xymatrix{
0 \ar[r] & \mathcal{H}^d_{(X)}(\mathcal{F})(U) \ar[r] \ar[d] & \bigoplus_{x \in U^{(0)}} \h_x^d(X,\mathcal{F}) \ar[r] \ar[d] & \bigoplus_{x \in U^{(1)}} \h_x^{d+1}(X,\mathcal{F}) \ar[r] \ar[d] & \cdots \\
0 \ar[r] & \h^d(Y,i^{-1} \mathcal{F}) \ar[r] & \bigoplus_{x \in Y^{(0)}}\h_x^d(X, \mathcal{F}) \ar[r] & \bigoplus_{x \in Y^{(1)}} \h_x^{d+1}(X,\mathcal{F}) \ar[r] & \cdots
} \] dont les fl\`eches verticales $\bigoplus_{x \in U^{(r)}} \h_x^{d+r}(X,\mathcal{F}) \to \bigoplus_{x \in Y^{(r)}}  \h_x^{d+r}(X,\mathcal{F})$ sont les projections induites par les inclusions $\{x \in Y^{(r)} \} \subseteq \{ x \in U^{(r)} \}$ .

En passant \`a la limite pour $U \ni x_0$, on obtient le diagramme commutatif \[ \xymatrix{
0 \ar[r] & \mathcal{H}^d_{(X)}(\mathcal{F})_{x_0} \ar[r] \ar[d] & \bigoplus_{\underset{x_0 \in \overline{\{ x \}}}{x \in X^{(0)}}} \h_x^d(X,\mathcal{F}) \ar[r] \ar[d] & \bigoplus_{\underset{x_0 \in \overline{\{ x \}}}{x \in X^{(1)}}} \h_x^{d+1}(X,\mathcal{F}) \ar[r] \ar[d] & \cdots \\
0 \ar[r] & \h^d(Y,i^{-1} \mathcal{F}) \ar[r] & \bigoplus_{x \in Y^{(0)}}\h_x^d(X, \mathcal{F}) \ar[r] & \bigoplus_{x \in Y^{(1)}} \h_x^{d+1}(X,\mathcal{F}) \ar[r] & \cdots
} . \] Or $\mathcal{H}^d_{(X)}(\mathcal{F})_{x_0} \to \h^d(Y,i^{-1} \mathcal{F})$ est un isomorphisme (d'apr\`es le lemme \ref{lemme RW inverse} ci-dessous), et les autres fl\`eches verticales aussi car l'inclusion d'espaces topologiques $Y \subseteq X$ est exactement l'ensemble des $x \in X$ tels que $x_0 \in \overline{\{ x \}}$.
\end{proof}

\begin{lemma}
\label{passage limite}
Soit $X$ un sch\'ema et soit $x_0 \in X$ (point ensembliste). Soit $\mathcal{F}$ un faisceau \'etale sur $X$. Notons $Y= \spec(O_{X,x_0})$ et $i$ le morphisme $\spec(O_{X,x_0}) \to X$. Alors on a un isomorphisme canonique \[ \mathcal{H}^d_{(X)}(\mathcal{F})_{x_0} \overset{\sim}{\to} \h^d(Y,i^{-1} \mathcal{F}) \] o\`u le membre de gauche est la fibre de $\mathcal{H}^d_{(X)}(\mathcal{F})$ en $x_0$.
\end{lemma}

\begin{proof}
Il suffit de constater que les morphismes \[ \varphi_U : \h^d(U,\mathcal{F}) \to \h^d(Y,i^{-1} \mathcal{F}) \] pour $U$ ouvert de $X$ contenant $x_0$ induisent un isomorphisme \[ \underset{\underset{x \in U \subseteq X}{\longrightarrow}}{\mathrm{lim}} \, \h^d(U,\mathcal{F}) \overset{\sim}{\to} \h^d(Y,i^{-1} \mathcal{F}) . \] Se restreignant aux ouverts $U$ qui sont affines, le sch\'ema $Y=\spec(O_{X,x_0})$ s'interpr\`ete comme la limite projective des $U$, les morphismes de transitions \'etant des morphismes affines. De plus on a une identification $i^{-1} \mathcal{F} = f_U^{-1} \mathcal{F}_{|U}$. Dans cette situation, \cite[Exp. VII, Th. 5.7.]{SGA4} nous assure que les $\varphi_U$ induisent l'isomorphisme voulu.
\end{proof}

Int\'eressons-nous maintenant plus sp\'ecifiquement aux complexes de faisceaux $\mathbb{Q}/\mathbb{Z}(j)$. Soit $0 \leqslant j \leqslant d$ un entier. On a
\begin{thm}[{Globalisation de \cite[Th. 4.1.]{Shiho}}]
\label{thm gersten}
Soit $X$ une $k$-vari\'et\'e lisse \'equidimensionnelle ou le spectre d'un anneau local en un point d'une $k$-vari\'et\'e lisse \'equidimensionnelle. Alors la conjecture de Gersten vaut pour $X$ et $\mathbb{Q}/\mathbb{Z}(j)$, c'est-\`a-dire qu'on a une suite exacte de faisceaux zariskiens sur $X$ : \begin{equation} 
\label{résolution gersten} 0 \to \mathcal{H}^d_{(X)}(\mathbb{Q}/\mathbb{Z}(j)) \to \bigoplus_{x \in X^{(0)}} (i_x)_\ast \h^d_x (X,\mathbb{Q}/\mathbb{Z}(j)) \to \bigoplus_{x \in X^{(1)}} (i_x)_\ast \, \h^{d+1}_x(X, \mathbb{Q}/\mathbb{Z}(j)) \to \cdots  .
\end{equation}
\end{thm}

\begin{proof}
Pour $j \geqslant 1$, on d\'ecompose $\mathbb{Q}/{\mathbb{Z}}(j)$ en somme directe de la limite inductive des $\mu_{q^n}^{\otimes j}$ pour $q \neq p$ premier, et de la limite inductive des $\nu_n^j[-j]$, $n \in \mathbb{N}$. Pour tout $x_0 \in X$, on note $Y_{x_0}$ le $k$-sch\'ema $\spec(O_{X,x_0})$.

$\diamond$ Pour tout $x_0 \in X$, le fait que le complexe des sections globales de $G^d(Y_{x_0},\mu_{q^n}^{\otimes j})$ soit exact est justifi\'e par \cite[Prop. 2.1.2. \& Th. 2.2.1.]{CT-H-K} : on a une suite exacte \[ 0 \to  \h^{d}(Y_{x_0}, \mu_{q^n}^{\otimes j}) \to \bigoplus_{x \in Y_{x_0}^{(0)}} \h_x^{d}(Y_{x_0},\mu_{q^n}^{\otimes j}) \to \bigoplus_{x \in Y_{x_0}^{(1)}} \h^{d+1}_x (Y_{x_0},\mu_{q^n}^{\otimes j}) \to \cdots .\] D\`es lors, gr\^ace au lemme \ref{lemma fibres gersten}, on conclut que $G^d(X,\mu_{q^n}^{\otimes j})$ est une suite exacte de faisceaux.

$\diamond$ On proc\`ede de m\^eme pour le cas avec les $\nu_n^j$. Pour tout $x_0 \in X$, le th\'eor\`eme \cite[Th. 4.1.]{Shiho} \'etablit que la suite de groupes suivante est exacte : \[ 0 \to  \h^{d-j}(Y_{x_0}, \nu_n^j) \to \bigoplus_{x \in Y_{x_0}^{(0)}} \h_x^{d-j}(Y_{x_0},\nu_n^j) \to \bigoplus_{x \in Y_{x_0}^{(1)}} \h^{d-j+1}_x (Y_{x_0},\nu_n^j) \to \cdots . \] D'apr\`es le lemme \ref{lemma fibres gersten}, cela signifie que les fibres en tout $x_0 \in X$ du complexe $G^{d-j}(X,\nu_n^j)$ sont exactes. On en d\'eduit que le complexe global est exact lui-aussi.
\end{proof}

\begin{rem}
Kahn montre que la conjecture de Gersten vaut pour les sch\'emas r\'eguliers de type fini sur un corps (\cite[Prop. A.4.]{Kahn_classes}). On pourrait donc justifier le th\'eor\`eme \ref{thm gersten} pour les anneaux locaux \`a partir du cas global.
\end{rem}

\begin{csq}
\label{csq gersten}
Soit $X$ une $k$-vari\'et\'e lisse et irr\'eductible ou le spectre d'un anneau local d'une $k$-vari\'et\'e lisse et irr\'eductible. On a une suite exacte de groupes \[ 0 \to  \h^0_\mathrm{Zar} \left( X,\mathcal{H}^d_{(X)}(\mathbb{Q}/\mathbb{Z}(j)) \right) \to \h^d (K(X),\mathbb{Q}/\mathbb{Z}(j)) \to \bigoplus_{x \in X^{(1)}} \h^{d+1}_x (X,\mathbb{Q}/\mathbb{Z}(j)) \] avec $K(X)$ le corps de fonctions de $X$ et une suite exacte pour tout $x_0 \in X$, \[ 0 \to  \h^d(O_{X,x_0}, \mathbb{Q}/\mathbb{Z}(j)) \to \h^d (K(X),\mathbb{Q}/\mathbb{Z}(j)) \to \bigoplus_{\underset{x_0 \in \overline{\{ x \}}}{x \in X^{(1)}}} \h^{d+1}_x (X,\mathbb{Q}/\mathbb{Z}(j)) . \]
\end{csq}

La premi\`ere suite s'obtient en prenant les sections globales dans le th\'eor\`eme \ref{thm gersten}, et la deuxi\`eme suite est \cite[Th. 4.1.]{Shiho}.

\vspace{\baselineskip}

Dans les hypoth\`eses de \ref{csq gersten}, on a un diagramme commutatif \`a lignes exactes, naturel en $(X,x_0)$, avec $x_0 \in X$ : \begin{equation}
\xymatrix{
0 \ar[r] &  \h^0_\mathrm{Zar} \left( X,\mathcal{H}^d_{(X)}(\mathbb{Q}/\mathbb{Z}(j)) \right) \ar[r] \ar[d] & \h^d (K(X),\mathbb{Q}/\mathbb{Z}(j)) \ar[r] \ar[d]^{=} &  \bigoplus_{x \in X^{(1)}} \h^{d+1}_{x} (X,\mathbb{Q}/\mathbb{Z}(j)) \ar[d]^{\text{projection}} \\
 0 \ar[r] &  \h^d ( O_{X,x_0},\mathbb{Q}/\mathbb{Z}(j) ) \ar[r] & \h^d (K(X),\mathbb{Q}/\mathbb{Z}(j)) \ar[r] & \bigoplus_{\underset{x_0 \in \overline{\{ x \}}}{x \in X^{(1)}}} \h^{d+1}_{x} (X,\mathbb{Q}/\mathbb{Z}(j))
},
\end{equation} duquel on d\'eduit : 
\begin{corol}
\label{corollaire points codimension 1}
Dans le groupe $\h^d (K(X),\mathbb{Q}/\mathbb{Z}(j))$, on a  l'\'egalit\'e entre sous-groupes \[ \h^0_\mathrm{Zar} \left( X,\mathcal{H}^d_{(X)}(\mathbb{Q}/\mathbb{Z}(j)) \right) = \bigcap_{x \in X^{(1)}} \h^d ( O_{X,x},\mathbb{Q}/\mathbb{Z}(j) ) . \]
\end{corol}

		\section{Groupe de Brauer (cohomologique)}
		\label{sous section groupe de brauer}

Ici on veut \'etablir clairement un isomorphisme naturel $\h^2(X,\mathbb{Q}/\mathbb{Z}(1)) \overset{\sim}{\to} \h^2(X,\mathbb{G}_m)$ pour certains sch\'emas $X$. La notation $\br^\prime(X)$ d\'esigne le groupe de Brauer cohomologique de $X$, c'est-\`a-dire le sous-groupe de torsion de $\h^2(X,\mathbb{G}_m)$, qui est $\h^2(X,\mathbb{G}_m)$ tout entier lorsque $X$ est irr\'eductible et r\'egulier noeth\'erien (\cite[Cor. 1.8.]{Grothendieck_brauerii}).

Soit $X$ un $\mathbb{F}_p$-sch\'ema r\'egulier avec $p$ un nombre premier. Les suites exactes \eqref{suite exacte logarithmique} fournissent des suites exactes en cohomologie \'etale \[ \h^1(X,\mathbb{G}_m) \overset{p^n}{\to} \h^1(X,\mathbb{G}_m) \to \h^1(X,\nu_n^1) \to \h^2(X,\mathbb{G}_m) \overset{p^n}{\to} \h^2(X,\mathbb{G}_m) \] pour tout entier $n \geqslant 1$, d'o\`u des suites exactes \[ 0 \to \pic(X)/p^n \pic(X) \to \h^2(X,\mathbb{Z}/{p^n\mathbb{Z}}(1)) \to \prescript{}{p^n}{\br^\prime(X)} \to 0 \] o\`u $\prescript{}{p^n}{\br^\prime(X)}$ d\'esigne les \'elements de $\br(X)$ qui s'annulent \`a la puissance $p^n$. Ces suites fournissent des morphismes naturels \[ \h^2(X,\mathbb{Z}/{p^n\mathbb{Z}}(1)) \to \prescript{}{p^n}{\br^\prime(X)} \] qui sont des isomorphismes quand $X$ est le spectre d'un anneau local, ou encore un morphisme naturel \begin{equation} \label{morphisme p brauer} \h^2(X,\mathbb{Q}_p/{\mathbb{Z}_p}(1)) \to \br^\prime(X) \lbrace p \rbrace 
\end{equation} qui est un isomorphisme quand $X$ est le spectre d'un anneau local. Pr\'ecisons que $\br^\prime(X)\lbrace p \rbrace$ est le sous-groupe de $\br^\prime(X)$ constitu\'e des \'el\'ements d'ordre une puissance de $p$. On compl\`ete \ref{morphisme p brauer} avec les parties de torsion premi\`ere \`a $p$ : \begin{equation}
\label{morphisme brauer}
\xymatrix{
\left( \underset{q \neq p \text{ premier}}{\bigoplus} \h^2(X,\mathbb{Q}_q/{\mathbb{Z}_q}(1)) \right) \oplus \h^2(X,\mathbb{Q}_p/{\mathbb{Z}_p}(1) ) \ar[r] \ar[d]^{=} & \left( \underset{q \neq p \text{ premier}}{\bigoplus} \br^\prime(X)\lbrace q \rbrace \right) \oplus \br^\prime(X)\lbrace p \rbrace \ar[d]^{=} \\
\h^2(X,\mathbb{Q}/{\mathbb{Z}}(1)) \ar@{-->}[r] &  \br^\prime(X)
}
\end{equation} pour obtenir un morphisme \emph{naturel}, qui est un isomorphisme sur la partie de torsion premi\`ere \`a $p$, et qui est un isomorphisme pour le spectre d'un anneau local.

%Pour r\'esumer, on a d\'efini un morphisme de groupes $\h^2(X,\mathbb{Q}/{\mathbb{Z}}(1)) \to \br^\prime(X)$ naturel en le $\mathbb{F}_p$-sch\'ema r\'egulier $X$.% Restreignons-nous au cas o\`u $X$ est soit une $k$-vari\'et\'e lisse irr\'eductible, soit le spectre d'un anneau local d'une telle vari\'et\'e.

\begin{thm}
\label{thm isom brauer}
On a un isomorphisme naturel en tout $\mathbb{F}_p$-sch\'ema irr\'eductible $X$ qui est une vari\'et\'e lisse sur un corps de caract\'eristique $p$ ou le spectre d'un anneau local d'une telle vari\'et\'e  : \[ \h^0_\mathrm{Zar} \left( X,\mathcal{H}^2_{(X)}(\mathbb{Q}/\mathbb{Z}(1)) \right) \overset{\sim}{\longrightarrow} \br(X). \]
\end{thm}

On montre d'abord le lemme suivant avant de prouver le th\'eor\`eme \ref{thm isom brauer}.

\begin{lemma}
\label{lemme faisceau brauer}
Soit $X$ un sch\'ema noeth\'erien, irr\'eductible et r\'egulier. Alors le pr\'efaisceau $U \mapsto \br^\prime(U) (=\h^2(U,\mathbb{G}_m))$ sur $X$ pour la topologie de Zariski est un faisceau.
\end{lemma}

\begin{proof}
Soit $U$ un ouvert non vide de $X$ et soit $(U_i)_{i \in I}$ un recouvrement ouvert de $U$. On veut montrer l'exactitude de la suite \begin{equation}
\label{suite faisceau} 0 \to \h^2(U,\mathbb{G}_m) \to \prod_{i \in I} \h^2(U_i,\mathbb{G}_m)  \rightrightarrows \prod_{i,j \in I} \h^2(U_i \cap U_j,\mathbb{G}_m) .
\end{equation}  L'ouvert $U$ \'etant non vide, $I$ est non vide et il existe $i_0 \in I$. D'apr\`es \cite[1.6.]{Grothendieck_brauerii}, les morphismes $\h^2(U,\mathbb{G}_m) \to \h^2(k(U),\mathbb{G}_m)$ et $\h^2(U_{i_0},\mathbb{G}_m) \to \h^2(k(U),\mathbb{G}_m)$ sont injectifs. Il faut donc que $\h^2(U,\mathbb{G}_m) \to \h^2(U_{i_0},\mathbb{G}_m)$ soit injectif, montrant par l\`a qu'on a bien l'injectivit\'e dans la suite \eqref{suite faisceau}. Ensuite, soit $(\alpha_i)_i \in \prod_{i \in I} \h^2(U_i,\mathbb{G}_m)$ tel que \[ \forall i, \, j \in I, \: (\alpha_i)_{|U_i \cap U_j} = (\alpha_j)_{|U_i \cap U_j} \in \h^2(U_i \cap U_j,\mathbb{G}_m). \] On consid\`ere un sous-ensemble fini $I^\prime \subseteq I$ pour lequel $(U_i)_{i \in I^\prime}$ reste un recouvrement de $U$. Cela est possible puisque $X$, donc $U$, est noeth\'erien. On veut montrer que les $\alpha_i$, $i \in I^\prime$, proviennent d'une m\^eme classe de $ \h^2(U,\mathbb{G}_m)$. Exprimons $I^\prime = \lbrace 1,\cdots,r \rbrace$ pour un entier $r \geqslant 1$. Pour $U_{1,2} := U_1 \cup U_2$, on a la suite exacte de Mayers-Vietoris (\cite[Th. 10.8.]{Milne_LectureE})) : \[  \h^2(U_{1,2},\mathbb{G}_m) \to  \h^2(U_1,\mathbb{G}_m) \oplus  \h^2(U_2,\mathbb{G}_m) \to  \h^2(U_1 \cap U_2,\mathbb{G}_m) . \] Par hypoth\`ese sur les $\alpha_i$, $i \in I$, les \'el\'ements $\alpha_1$ et $\alpha_2$ ont m\^eme image dans $\h^2(U_1 \cap U_2,\mathbb{G}_m) $, donc proviennent d'une classe $\alpha_{1,2} \in \h^2(U_{1,2},\mathbb{G}_m)$. En proc\'edant de m\^eme avec $\alpha_{1,2}$ et $\alpha_3$ \`a la place de $\alpha_1$ et $\alpha_2$, et ainsi de suite, on trouve qu'effectivement $\alpha_1$, $\cdots$, $\alpha_r$ proviennent d'un m\^eme \'el\'ement de $\h^2(U,\mathbb{G}_m)$.
\end{proof}

\begin{proof}[D\'emonstration du th\'eor\`eme \ref{thm isom brauer}]
Appelons $\varphi_Y$ le morphisme \ref{morphisme brauer} pour un sch\'ema $Y/\mathbb{F}_p$ r\'egulier irr\'eductible. 

Soit $X$ comme dans le th\'eor\`eme. Du fait que le groupe de Brauer est un faisceau sur $X$ pour la topologie de Zariski (lemme \ref{lemme faisceau brauer}), les $\varphi_U$, pour $U$ ouverts de $X$, induisent un morphisme de faisceaux zariskiens $\phi : \mathcal{H}^2_{(X)}(\mathbb{Q}/\mathbb{Z}(1)) \to \br^\prime$.

Pour tout $x \in X$, on a un diagramme commutatif \begin{equation}
\xymatrix{
\h^0_\mathrm{Zar} \left( X,\mathcal{H}_{(X)}^2(\mathbb{Q}/\mathbb{Z}(1)) \right) \ar[r]^-{\phi(X)} \ar[d] & \br^\prime(X) \ar[d] \\
\h^2 ( O_{X,x},\mathbb{Q}/\mathbb{Z}(1)) \ar[r]^-{\varphi_{O_{X,x}}}_-{\simeq} \ar[d] & \br^\prime(O_{X,x})\ar[d] \\
\h^2 ( K(X),\mathbb{Q}/\mathbb{Z}(1)) \ar[r]^-{\varphi_{K(X)}}_-{\simeq} & \br^\prime(K(X))
}.
\end{equation} De ce diagramme, gr\^ace au corollaire \ref{corollaire points codimension 1} et du fait que \[ \br^\prime(X)= \bigcap_{x \in X^{(1)}} \br^\prime(O_{X,x}) \subset \br^\prime(K(X))\] (d'apr\`es \cite[Th. 1.2]{Cesnavicius}), on tire que $\phi(X)$ est un isomorphisme, qui est naturel en $X$.
\end{proof}

\bibliographystyle{alpha-fr}
\bibliography{/Users/alexlourdeaux/Documents/Mathematiques/Perso/Bibliographie.bib}

\end{document}